\documentclass[12pt]{amsart}

 \newtheorem{thm}{Theorem}[section]
 \newtheorem{cor}[thm]{Corollary}
 \newtheorem{lem}[thm]{Lemma}
 \newtheorem{prop}[thm]{Proposition}
 \theoremstyle{definition}
 
 \theoremstyle{remark}
 \newtheorem{rem}[thm]{Remark}
 \theoremstyle{definition}
 \newtheorem{ex}[thm]{Example}

 \newcommand{\CC}{\mathbb{C}}

 \newcommand{\XX}{\mathbb{X}}
 \newcommand{\PP}{\mathbb{P}}
\newcommand{\Z}{\mathbb{Z}}
\newcommand{\X}{\mathbb{X}}

\usepackage{xypic}

\begin{document}

\title{Complete intersections on general hypersurfaces}
\author[E. Carlini]{Enrico Carlini}
\address[E. Carlini]{Dipartimento di Matematica, Politecnico di Torino, Torino, Italia}
\email{enrico.carlini@polito.it}

\author[L. Chiantini]{Luca Chiantini}
\address[L. Chiatini]{Dipartimento di Scienze Matematiche e Informatiche, Universit\`{a} di Siena, Siena, Italia}
\email{chiantini@unisi.it}

\author[A.V. Geramita]{Anthony V. Geramita}
\address[A.V. Geramita]{Department of Mathematics and Statistics, Queen's University, Kingston, Ontario, Canada, K7L 3N6 and Dipartimento di Matematica, Universit\`{a} di Genova, Genova, Italia}
\email{Anthony.Geramita@gmail.com \\ geramita@dima.unige.it  }

\address{}

\email{}

 \subjclass{Primary ; Secondary }

\date{}

%%% ----------------------------------------------------------------------

\begin{abstract}
We ask when certain complete intersections of codimension $r$ can
lie on a generic hypersurface in $\PP^n$.  We give a complete
answer to this question when $2r \leq n+2$ in terms of the degrees
of the hypersurfaces and of the degrees of the generators of the
complete intersection.

\end{abstract}

%%% ----------------------------------------------------------------------
\maketitle
%%% ----------------------------------------------------------------------

\section{Introduction}

Many problems in classical projective geometry ask about the
nature  of special subvarieties of some given family of varieties,
e.g. how many isolated singular points can a surface of degree $d$
in $\PP^3$ have? when is it true that the members of a certain
family of varieties contain a line? contain a linear space of any
positive dimension?  The reader can easily supply other examples
of such questions.

This is the kind of problem we consider in this paper: what  types
of complete intersection varieties of codimension $r$ in $\PP^n$
can one find on the generic hypersurface of degree $d$?

In case $r=2$ it was known to Severi \cite{Severi} that for  $n
\geq 4$ the only complete intersections on a general hypersurface
are obtained by intersecting that hypersurface with another.

This observation was extended to $\PP^3$ by Noether (and
Lefschetz)  \cite{Lefschetz, GH85} for general hypersurfaces of
degree $\geq 4$.  These ideas were further generalized by
Grothendieck \cite{SGA2}.

Our approach to the problem mentioned above uses a mix of
projective  geometry and commutative algebra and is much more
elementary and accesible than, for example, the approach of
Grothendieck.  We are able to give a complete answer to the
question we raised for complete intersections of codimension $r$
in $\PP^n$ which lie on a general hypersurface of degree $d$
whenever $2r \leq n+2$.

The paper is organized in the following way: in the next section
(Section 2) we lay out the question we want to consider and
explain what are the interesting parameters for a response.

In Section 3 we collect some technical information we will  need
about varieties of reducible forms and their joins.  In order to
find the dimensions of these joins (using \lq\lq Terracini's
Lemma") we calculate the tangent space at a point of any variety
of reducible forms.  We also recall some information about
artinian complete intersection quotients of a polynomial ring.

In Section 4, we use the technical facts collected in Section 3
to reformulate our original question.  We illustrate the utility
of this reformulation to discuss complete intersections of
codimension $r$ in $\PP^n$ on a general hypersurface when $2r <
n+1$.  We further use our approach to give a new proof for the
existence of a line on the general hypersextic of $\PP^5$.

In Section 5 we state and prove our main theorem  which gives a
complete description of all complete intersections of codimension
$r$ in $\PP^n$ which lie on a generic hypersurface when $2r \leq
n+2$.

\section{Question}\label{questionSECTION}

The  objects of study of this paper are complete intersection
subschemes of projective space.  Recall that $Y\subset\PP^n$ is a
complete intersection scheme  if its ideal is generated by a
regular sequence, more precisely, $I(Y)=(F_1,\ldots,F_r)$, $F_i\in
S=\mathbb{C}[x_0,\ldots,x_n]$, and $F_1, \ldots , F_r$ form a
regular sequence in $S$. If $\deg F_i=a_i$ for all $i$, we will
say that such a $Y$ is a $CI(a_1,\ldots,a_r)$ and we will assume
$a_1\leq\ldots\leq a_r$; notice that $Y$ is unmixed of codimension
$r$ in $\PP^n$. With this notation we can rephrase the statement

\begin{quote} the degree $d$ hypersurface $X$ contains a
$CI(a_1,\ldots,a_r)$
\end{quote}

in terms of ideals as follows

\begin{quote}
$I(X)=(F)\subset (F_1,\ldots,F_r)$ for forms $F_i$ forming a
regular sequence and such that $\deg F_i=a_i$ for all $i$.
\end{quote}

Clearly not all choices of the degrees are of interest for us,
e.g. if $a_i>d$ for all $i$, then no $CI(a_1,\ldots,a_r)$ can be
found on a degree $d$ hypersurface. On the other hand, any
hypersurface of degree $d$ contains a $CI(a_1,\ldots,a_r)$ if
$a_i=d$, for some $i$. Simply cut that hypersurface with general
hypersurfaces of degrees $a_j,\ j\neq i$.

So, one need only consider $CI(a_1,\ldots , a_r)$ where none of the $a_i = d$.

\begin{lem}\label{lessthand} Let $a_1 \leq \ldots \leq a_i < d < a_{i+1} \leq \ldots \leq a_r$ with $r \leq n$.
The following are equivalent facts:

\begin{itemize}
\item there is a $CI(a_1, \ldots, a_r)$ on the general
hypersurface of degree $d$ in $\PP^n$;

\item there is a $CI(a_1, \ldots , a_i)$ on the general
hypersurface of degree $d$ in $\PP^n$.
\end{itemize}
\end{lem}

\begin{proof}
Let $I(X)=(F)$, where $X$ is a general hypersurface of  degree $d$
in $\PP^n$.  Let $I(Y)=(F_1,\ldots,F_r)$ be the ideal of a
$CI(a_1, \ldots , a_r)$, with degrees $a_i$ as above.

Then $X\supset Y$ if and only if
\[
F=\sum_{j=1}^i F_jG_j
\]
and hence, if and only if $X\supset Y^\prime$, where $Y^\prime$ is
the complete intersection defined by $F_1, \ldots , F_i$.
\end{proof}

From this Lemma it is clear that the basic question to be
considered is:
\begin{quote}{(Q): }{\em
For which degrees $a_1,\ldots, a_r<d$ does the generic degree $d$
hypersurface of $\PP^n$ contain a $CI(a_1,\ldots,a_r)$?
}\end{quote}

If, rather than restricting to the generic case, we asked if {\em
some} hypersurface of degree $d$ contains a $CI(a_1, \ldots ,
a_r)$, then the answer is trivial. Indeed,  the ideal of any
$CI(a_1,\ldots,a_r)$ ($a_i < d$) always contains degree $d$
elements.

%Indeed, there are degree $d$ hyperurfaces containing a $CI(a_1,
%\ldots ,a_r)$ for any choice of the degrees $a_i < d$. In fact,
%the ideal of such a $CI(a_1,\ldots,a_r)$ has degree $d$ elements.
%Thus, to give an answer to our question we have to determine how
%special these hyperusrfaces are among the degree $d$
%hypersurfaces.

\section{Technical facts}

If $\lambda = (\lambda_1, \ldots , \lambda_s)$ is a  partition of
the integer $d$ (i.e. $\sum_{i=1}^s\lambda_i = d$ and $\lambda_1
\geq \ldots \geq \lambda _s > 0$) we write $\lambda\vdash d$.  For
each $\lambda \vdash d$ we define a subvariety $\XX_\lambda
\subset \PP (S_d) \simeq \PP^N$ (where $N = {d+n \choose n} -1$)
as follows:
\[
\XX_\lambda:= \{\ \ [F] \in S_d \ \vert \ F=F_1\cdots F_s, \ \deg F_i = \lambda_i \ \ \}.
\]
We call $\XX_\lambda$ the {\it variety of reducible forms of type
$\lambda$}.  The dimension of $\XX_\lambda$ is easily seen to be
$[\sum_{i=1}^s {\lambda_i+n\choose n}]-s$. (For other elementary
properties of $\XX_\lambda$ see \cite{Mamma} and for the special
case $\lambda_1=\ldots=\lambda_s=1$ see \cite{Ca04JA},
\cite{Ca05Siena} or \cite{Chi02} for the $n=2$ case).

If $x_1, \ldots , x_r$ are independent points  of $\PP^N$ we will
call the $\PP^{r-1}$ spanned by these points the {\it join of the
points $x_1, \ldots x_r$} and write
\[
J(x_1, \ldots , x_r) := \langle x_1, \ldots , x_r \rangle .
\]

More generally, if $X_1, \ldots , X_r$ are  varieties in $\PP^N$
then the {\it join of $X_1, \ldots , X_r$} is
\[
J(X_1, \ldots , X_r):= \overline{ \bigcup \{ J(x_1, \ldots , x_r)\ \vert \ x_i \in X_i, \{x_1, \ldots , x_r\}\ \hbox{independent} \}}
\]

In case $X_1 = \cdots =  X_r = X$ we write
\[
J(X_1, \ldots , X_r):= Sec_{r-1}(X)
\]
and call this the {\it $(r-1)^{st}$ (higher) secant variety of
$X$}.

Joins and secants of projective  varieties are important auxiliary
varieties which can help us better understand the geometry of the
original varieties  (see e.g.
\cite{CGG1,CGG2,ChCi06,Ci01,Ge,LM,AAd,CatJ}). One of the most
fundamental questions we can ask about joins and secants is: {\it
What are their dimensions?}

This is, in general, an extremely  difficult question to answer.
The famous {\it Lemma of Terracini} (which we recall below) is an
important observation which will aid us in answering this
question.

\begin{lem}
[Lemma of Terracini]\label{terracinilemma}   Let $X_1, \ldots ,
X_r$  be reduced subvarieties of $\PP^N$ and let $p \in J = J(X_1,
\ldots , X_r)$ be a generic point of $J$.

Suppose that $p \in J(p_1, \ldots , p_r)$,  then the
(projectivized) tangent space to $J$ at $p$, i.e. $T_p(J)$, can be
described as follows:
\[
T_p(J) = \langle T_{p_1}(X_1), \ldots , T_{p_r}(X_r) \rangle .
\]

Consequently,
\[
\dim J = \dim \langle T_{p_1}(X_1), \ldots , T_{p_r}(X_r) \rangle .
\]

\end{lem}

We want to apply this Lemma in the case that  the $X_i$ are all of
the form $\XX_{\lambda^{(i)}}, \ \lambda^{(i)} \vdash d, \ i = 1,
\ldots , r$.  A crucial first step in such an application is,
therefore, a calculation of $T_{p_i}(\XX_{\lambda ^{(i)} })$ where
$p_i \in \XX_{\lambda ^{(i)} }$.

\begin{prop}\label{tangentspace}
Let $\lambda \vdash d$, $\lambda = (\lambda_1, \ldots , \lambda_s)$ and let $p \in \XX_\lambda$ be a generic point of $\XX_\lambda$.

Write $p = [F_1\cdots F_s]$ where $\deg F_i = \lambda_i, i = 1,
\ldots , s$ and let $I_p\subset S=\CC[x_0, \ldots , x_n]$ be the
ideal defined by:
\[
I_p = (F_2\cdots F_s, F_1F_3\cdots F_s, \ldots , F_1 \cdots F_{s-1} ).
\]

Then the tangent space to $\XX_\lambda$ at the point $p$ is the projectivization of $(I_p)_d$ and hence has dimension
\[
\dim T_p(\XX_\lambda) = \dim _\CC (I_p)_d - 1.
\]
\end{prop}

\begin{proof}
Consider the map of affine spaces
\[
\Phi: S_{\lambda_1} \times \cdots \times  S_{\lambda_s} \rightarrow S_d
\]
defined by
\[
\Phi((A_1, \ldots , A_s)) =  A_1\cdots A_s .
\]

Let $P \in S_{\lambda_1} \times \cdots \times  S_{\lambda_s}$ be
the  point $P = (F_1, \ldots , F_s)$.  A tangent direction at $P$
is given by any vector of the form ${v} = (F_1^\prime, \ldots ,
F_s^\prime)$ and the line through $P$ in that direction is
\[
L_{{v}}: = (F_1 + \mu F_1^\prime, \ldots , F_s + \mu F_s^\prime),\
\mu \in \CC .
\]

A simple calculation shows that the tangent  vector to
$\Phi(L_{{v}})$ at the point $\Phi(P)=p$ is exactly $\sum_{i=1}^s
F_1\cdots  {F_i}^\prime  \cdots F_s$ and that proves the
proposition.

\end{proof}

In view of Terracini's Lemma, the following corollary is immediate.

\begin{cor} \label{tangentjoin}

Let $\lambda^{(1)}, \ldots , \lambda^{(r)}$ all be partitions of $d$ where
\[
\lambda^{(i)} = (\lambda_{i1}, \lambda_{i2}) .
\]

Let
\[
I = (F_{11}, F_{12}, F_{21}, F_{22}, \ldots , F_{r1}, F_{r2} )
\]
be an ideal of $S$ generated by generic forms where
\[
\deg F_{ij} = \lambda_{ij}, \hbox{ for } 1 \leq i \leq r, j = 1,2.
\]

If
\[
J = J(\XX_{\lambda^{(1)}}, \ldots , \XX_{\lambda^{(r)}})
\]
then
\[
\dim J = \dim _\CC I_d - 1.
\]
\end{cor}

\begin{rem} \label{genjoin} It is useful to note the
following facts:

\begin{enumerate}
\item In Proposition \ref{tangentspace} we are using the fact that
$\mathbb{C}$ has characteristic 0. The problem is that the
differential is not necessarily generically injective in
characteristic $p$.

\item\label{regseqrem1}   Observe that the generic  point in
$\XX_\lambda$, $\lambda = (\lambda_1, \ldots , \lambda_s) \vdash
d$ can always be written as the product of $s$ irreducible forms
with the property that any $\ell$-subset of these $s$ forms ($\ell
\leq n+1$) is a regular sequence.

\item\label{regseqrem3}  This last can be extended  easily to
joins of varieties of reducible forms.  I.e. the generic point in
such a join can be written as a sum of elements with the property
that each summand is a point enjoying the property described in
\eqref{regseqrem1} above. Moreover, every $\ell$-subset ($\ell
\leq n+1$) of the set of all the irreducible factors of all of
these summands is also a regular sequence.

\item Fr\"{o}berg (see \cite{Fro}) has a conjecture about the
multiplicative  structure of rings $S/I$, where $S = \CC[x_0,
\ldots , x_n]$ and $I$ is an ideal generated by a set of generic
forms.  This conjecture gives the Hilbert functions of such rings.
However, apart from the cases $n=1$ (proved several times by
various authors, see \cite{Fro, GeSc,
IaKa}) and $n=2$ (proved by \cite{Anick}) this conjecture
has resisted attempts to prove it.

Notice that in terms of the geometric  problem in Corollary
\ref{tangentjoin}, one need only consider Fr\"{o}berg's conjecture
for a strongly restricted collection of degrees.
\end{enumerate}
\end{rem}

We will need some specific information  about the Hilbert function
of some artinian complete intersections in polynomial rings.  The
following lemma summarizes the facts we shall use.

\begin{lem}\label{basiclemma} Let $r>1$ and $F_1,\ldots,F_r,G_1,\ldots,G_r$ be generic forms in $\CC[y_1,\ldots,y_{2r-1}]$ having degrees
\[1<\deg F_1=a_1\leq\deg F_2=a_2\leq\ldots\leq\deg F_r=a_r\leq d/2\]
and
\[d/2\leq\deg G_r=d-a_r\leq\ldots\leq\deg G_1=d-a_1\]
for a non-negative integer $d$.

Consider the quotient
\[A=\CC[y_1,\ldots,y_{2r-1}]/(F_1,\ldots,F_r,G_r,\ldots,G_{3},G_{2})\]
and its Hilbert function $H_A$. The following facts hold:
\begin{enumerate}

\item\label{one} $H_A$ is symmetric with respect to
$c={(r-1)d+a_1-2r+1\over 2}$;

\item\label{two} if $H_A(i)\geq H_A(i+1)$ then $H_A(j)$ is
non-increasing for $j\geq i$;

\item\label{six} the multiplication map on  $A_i$  given by
$\overline{G_1}$ (the class of ${G_1}$ in $A$) has maximal rank.

\end{enumerate}

If one of the following holds

\[r=2 \mbox{ and } a_1\geq 5, \mbox{ or}\]
\[r=3 \mbox{ and } a_1\geq 3, \mbox{ or}\]
\[r=3 \mbox{ and } a_1=2,d\neq 4, \mbox{ or}\]
\[r>3 \mbox{ and } a_1\geq 2,\]

we also have that:

\begin{enumerate}

\setcounter{enumi}{3}

\item\label{three} if $i\leq a_1$, then $H_A(i)<H_A(i+1)$.

\item\label{four} if $a_1<i\leq c$, then $H_A(a_1)<H_A(i)$.

\item\label{five} if $c<i$, then $H_A(a_1)>H_A(i)$ if and only if
$c-a_1<i-c$.

\end{enumerate}
\end{lem}

\begin{proof}
As $A$ is a Gorenstein graded ring \eqref{one}  follows
immediately, while \eqref{six} is a consequence of a theorem of
Stanley \cite{Stanley} and Watanabe \cite{Wat}.

To prove \eqref{two} we can use the Weak Lefschetz property, i.e.
multiplication by a general linear form  has maximal rank, e.g.
see \cite{MigMR}. The condition on $H_A$, coupled with the Weak
Lefschetz property, yields that every element of $A_{i+1}$ is the
product of a fixed linear form with a form of degree $i$. Now
consider an element of $A_{i+2}$, call it $M$ , then since $A$ is
a standard graded algebra, $M = \sum_{i=1}^{2r-1} y_iC_i$, where
$y_i$ is the class of $y_i$ in $A$, and $C_i$ is the class of a
form of degree $i+1$.  By what we have seen, $C_i = LD_i$ where $L
$ is the form we had earlier and the $D_i$ are forms of degree
$i$. Rewriting we get $M = L\sum_{i=1}^{2r-1}y_iD_i$.  But
$\sum_{i=1}^{2r-1}y_iD_i$ is in $A_{i+1}$ hence $A_{i+2} =
LA_{i+1}$ and hence the dimension cannot increase. Proceeding by
induction we prove the statement.

% is equal to $LN$
%for some form $N$ in $A_i$.  We conclude that $M = L(LN) = L^2 N$
%(note only one linear form is used.).  In particular $A_{i+2} =
%LA_{i+1}$ and hence the dimension cannot increase.

%The condition on $H_A$, coupled with the Weak Lefschetz property,
%yields that every element of $A_{i+1}$ is the product of a fixed
%linear form with a form of degree $i$. Then, we apply again the
%Weak Lefschetz property to pass from $A_{i+1}$ to $A_{i+2}$
%multiplying by another generic linear form. Hence any element of
%$A_{i+2}$ is the product of a fixed form of degree $2$ (namely the
%product of the two generic linear forms) with a form of degree
%$i$. Iterating this argument yields that
%\[H_A(j)=\dim A_j\leq H_A(i)=\dim A_i\]
%for $j\geq i$ as we wanted to show.

As for \eqref{three}, it is enough to give the proof for $i=a_1$
as there are no generators of degree smaller than $a_1$. Let $\bar
A$ be a quotient obtained when all the forms $F_i$ and $G_i$ have
the same degree $a = a_1 = \ldots =a_r=d-a_r=\ldots = d-a_2 $.
Notice that it is enough to show the result for $\bar A$. In fact,
whenever we pass from $\bar A$ to another quotient $A$ by
increasing the degrees of $s$ forms we obtain

\begin{equation}
\begin{array}{l}
\displaystyle H_A(a)=H_{\bar A}(a)+s\\ \\\displaystyle H_{\bar
A}(a+1)+s(2r-1)-s\leq H_A(a+1)
\end{array}
\label{HFestimates}
\end{equation}

and the inequality $H_{\bar A}(a)<H_{\bar A}(a +1)$ is preserved;
these Hilbert function estimates use the fact that the forms $F_i$
and $G_i$ do not have linear syzygies. By straightforward
computations one gets
\[H_{\bar A}(a)={a+2r-2\choose a}-2r+1 \]
and
\[H_{\bar A}(a+1)={a+2r-1\choose a+1}-(2r-1)^2.\]
Thus the inequality $H_{\bar A}(a)<H_{\bar A}(a+1)$ is equivalent
to

\begin{equation}
{a+2r-2\choose a+1}-(2r-1)(2r-2)>0. \label{star}
\end{equation}

Notice that if \eqref{star} holds for the pair $(a,r)$ then it
holds for all the pairs $(a+i,r)$ with $i\geq 0$. By direct
computations we verify that the inequality is satisfied for
$(a,r)=(5,2),(3,3)$ and for $a=2$ and $r>3$. Hence, \eqref{star}
holds for
\[r=2 \mbox{ and } a\geq 5, \mbox{ or }\]
\[r=3 \mbox{ and } a\geq 3, \mbox{ or }\]
\[r>3 \mbox{ and } a\geq 2.\]
To complete the proof of \eqref{three} it is enough to evaluate
\eqref{HFestimates} for $r=3,a=2$ in the case $d\neq 4$, i.e.
$s>0$.

To show \eqref{four}, notice that by \eqref{two}, if
\[H_A(a_1)>H_A(i), a_1<i\]
then $H_A$ is definitely non-increasing and hence it cannot be
symmetric with respect to $c$ by \eqref{three}.

To get \eqref{five} it is enough to use symmetry and \eqref{four}.

\end{proof}

\section{Equivalences}

In this section we give some equivalent formulations of our basic
question (Q), formulated at the end of Section
\ref{questionSECTION}.

Clearly, if $X\subset\PP^n$ is a hypersurface of degree $d$ and
$Y\subset X$ is a $CI(a_1,\ldots,a_r)$, then the ideal inclusion
$I(X)=(F)\subset I(Y)=(F_1,\ldots,F_r)$ yields
\[F=F_1G_1+\ldots +F_rG_r \]
for forms $G_i$ of degrees $d-a_i$. But the converse is not true
in general. If $F=F_1G_1+\ldots +F_rG_r$ and the forms $F_i$ do
not form a regular sequence, then $(F_1,\ldots,F_r)$ is not the
ideal of a complete intersection. To produce an equivalence we
need to use joins:

\begin{lem}\label{joineqlem} The following are equivalent:
\begin{enumerate}
\item\label{joineqlem1} a generic hypersurface of degree $d$ of
$\PP^n$ contains a $CI(a_1,\ldots,a_r)$, where $a_i<d$ for all
$i$;

\item\label{joineqlem2} the join of the varieties of reducible
forms
 $\mathbb{X}_{(a_i,d-a_i)},
i=1,\ldots,r$ fills the space of degree $d$ forms in $n+1$
variables, i.e.
\[J(\mathbb{X}_{(a_1,d-a_1)},\ldots,\mathbb{X}_{(a_r,d-a_r)})=\PP(S_d).\]
\end{enumerate}
\end{lem}
\begin{proof}
The implication \eqref{joineqlem1}$\Rightarrow$\eqref{joineqlem2}
simply follows from the ideal inclusion argument above yielding
the presentation $F=\sum F_iG_i$ for the generic degree $d$ form,
where $[F_iG_i]\in \mathbb{X}_{(a_i,d-a_i)}$ for all $i$. The
implication \eqref{joineqlem2}$\Rightarrow$\eqref{joineqlem1} is
easily shown using the description of the generic element of the
join, see Remark \ref{genjoin} \eqref{regseqrem3}.
\end{proof}
\begin{rem}\label{switch} Notice that there is an equality of varieties
\[\mathbb{X}_{(i,j)}=\mathbb{X}_{(j,i)}\]
for all non negative integers $i$ and $j$. Hence, by Lemma
\ref{joineqlem}, the condition
\[J(\mathbb{X}_{(a_1,d-a_1)},\ldots,\mathbb{X}_{(a_r,d-a_r)})=\PP(S_d)\]
is equivalent to the statement
\begin{quote}
a generic hypersurface of degree $d$ of $\PP^n$ contains a
$CI(b_1,\ldots,b_r)$, where $b_i=a_i$ or $b_i=d-a_i$ for all $i$.
\end{quote}
It follows from these observations that we can further restrict
the range of the degrees in our basic question (Q), i.e. it is
enough to consider
\[a_1\leq\ldots\leq a_r\leq {d\over 2}.\]
\end{rem}

Now we exploit Terracini's Lemma and the tangent space description
given in Corollary \ref{tangentjoin} in order to produce another equivalent
formulation of question (Q).
\begin{lem}\label{algeqlem} The following are equivalent:
\begin{enumerate}
\item\label{algeqlem1} the generic hypersurface of degree $d$ of
$\PP^n$ contains a $CI(a_1,\ldots,a_r)$, where $a_i<d$ for all
$i$;

\item\label{algeqlem2} let $F_1,\ldots,F_r$ and $G_1,\ldots,G_r$
be generic forms in $S=\mathbb{C}[x_0,\ldots,x_n]$ of degrees
$a_1\leq\ldots \leq a_r<d$ and $d-a_1,\ldots,d-a_r$ respectively,
then
\[H(S/(F_1,\ldots,F_r,G_r,\ldots,G_1),d)=0\]
where $H(\cdot,d)$ denotes the Hilbert function in degree $d$ of
the ring.
\end{enumerate}
\end{lem}
\begin{proof}
The condition on the join in Lemma \ref{joineqlem} can be read in
term of tangent spaces as equivalent to
\[\langle T_{P_1}(\mathbb{X}_{(a_1,d-a_1)}),\ldots,T_{P_r}(\mathbb{X}_{(a_r,d-a_r)}) \rangle=\PP(S_d)\]
for generic points $P_1=[F_1G_1],\ldots,P_r=[F_rG_r]$. Using the
description of the tangent space to the variety of reducible forms
this is equivalent to saying
\[(F_1,G_1)_d + \ldots+(F_r,G_r)_d=S_d\]
where $S_d$ is the degree $d$ piece of the polynomial ring $S$ and
the forms $F_i$ and $G_i$ are generic of degrees $a_i$ and $d-a_i$
respectively.
\end{proof}

As a straightforward application we get the following result:
\begin{prop}\label{smallrPROP}
The generic degree $d$ hypersurface of $\PP^n$ contains no
$CI(a_1,\ldots,a_r)$, $a_i<d$ for all $i$, when $2r<n+1$.
\end{prop}
\begin{proof}
We use Lemma \ref{algeqlem}. Consider in
$S=\mathbb{C}[x_0,\ldots,x_n]$ generic forms $F_1,\ldots,F_r$ and
$G_1,\ldots,G_r$ of degrees $a_i$ and $d-a_i$ respectively. If we
let $I$ be the ideal $(F_1,\ldots,F_r,G_1,\ldots,G_r)$, then we
want to show that $H(S/I,d)\neq 0$ and for that it is enough to
show that $S/I$ is not an artinian ring. As $I$ has height $2r$
and $2r<n+1$ the quotient cannot be zero dimensional and the
conclusion follows.
\end{proof}
\begin{rem} Using Lemma \ref{algeqlem} we can also recover many
classical results in an elegant and simple way. More precisely, we
can easily study the existence of complete intersections curves,
e.g. lines and conics, on hypersurfaces.

\end{rem}

\begin{ex} As an example we prove the following without using Schubert calculus:

{\em The generic hypersextic of $\PP^5$ contains a line.}
\begin{proof} Let $S=\mathbb{C}[x_0,\ldots,x_5]$ and consider the
ideal

\[I=(L_1,\ldots ,L_4,G_1,\ldots,G_4)\]

where the forms $L_i$ are linear forms and the forms $G_i$ have
degree $5$. We want to show that $H(S/I,6)=0$.  Clearly
\[ S/I \simeq \mathbb{C}[x_0,x_1]/(\bar G_1,\ldots,\bar G_4).\]

It is well known \cite{GeSc, IaKa,
Fro}  that 4 general binary forms of degree 5 generate $\CC
[x_0,x_1]_6$ and we are done.
\end{proof}

For more on this topic see Remark \ref{hochesterRem}.
\end{ex}

\section{The Theorem}

We are now ready to prove the main  theorem of this paper, a
description of all the possible complete intersections of
codimension $r$ that can be found on a general hypersurface of
degree $d$ in $\PP^n$ when $2r \leq n+2$.

%Recall that from Lemma \ref{lessthand} and Remark \ref{switch} it
%is  sufficient to consider the existence of a $CI(a_1, \ldots ,
%a_r)$ on the generic hypersurface of degree $d$ when $a_1\leq
%\ldots \leq  a_r\leq d/2$.

\begin{thm}\label{maintheorem}
Let $X\subset\PP^n$ be a generic degree $d$ hypersurface, with
$n,d>1$. Then $X$ contains a $CI(a_1,\ldots,a_r)$, with $2r\leq
n+2$, and the $a_i$ all less than $d$, in the following (and only
in the following) instances:
\medskip

\begin{itemize}
\item $n=2$: then $r=2$, $d$ arbitrary and $a_1$ and $a_2$ can
assume any value less than $d$;

\item $n=3$, $r=2 $:  for $d \leq 3$ we have that $a_1$ and $a_2$
can assume any value less than $d$;

\item $n=4$, $r = 3$: for $d \leq 5$ we have that $a_1, a_2$ and
$a_3$ can assume any value less than $d$;

\item $n=6,r = 4$ or $n=8,r = 5$: for $d \leq 3$ we have that
$a_1, \ldots ,a_r$ can assume any value less than $d$;

\item $n=5,7$ or $n>8$, $2r=n+1$ or $2r = n+2$:  we have only
linear spaces on quadrics, i.e. $d = 2$ and $a_1 = \ldots =a_r =
1$.

\end{itemize}

\end{thm}
\begin{proof}
 Recall that from Lemma \ref{lessthand} and Remark \ref{switch} it
 is  sufficient to consider the existence of a $CI(a_1, \ldots ,
 a_r)$ on the generic hypersurface of degree $d$ when $a_1\leq
 \ldots \leq  a_r\leq d/2$.

When $2r<n+1$, by Proposition \ref{smallrPROP}, we know that no
complete intersection exists. Hence we have only to consider the
cases $2r=n+1$ and $2r=n+2$

In order to use Lemma \ref{algeqlem} we consider the generic forms
$F_1,\ldots,F_r$ and $G_1,\ldots,G_r$ of degrees $a_i$ and $d-a_i$
respectively. If we let $S=\CC[x_0,\ldots,x_n]$ and
$I=(F_1,\ldots,F_r,G_r,\ldots,G_1)$ we want to check whether
$H(S/I,d)=0$ or not.

The case $2r=n+1$. In this case, $S/I$ is an artinian Gorenstein
ring and $e=r(d-2)+1$ is the first place where one has
$H(S/I,e)=0$. Thus, the generic degree $d$ hypersurface contains a
$CI(a_1,\ldots,a_r)$ if and only if $H(S/I,d)=0$ and this is
equivalent to the inequality
\[d\geq r(d-2)+1\]
which is never satisfied unless $d=2$ and $a_1=\ldots=a_r=1$.

The case $2r=n+2$ will be proved using Lemma \ref{basiclemma}. In
order to do this, we divide the proof into three parts:

\begin{itemize}

\item {\em the hyperplane case}: $a_1=1$ any $r$;

\item {\em the plane case}: $a_1=2,3,4$ for $r=2$, and hence
$n=2$;

\item {\em the four space case}: $a_1=a_2=a_3=2$ and $d=4$ for
$r=3$, and hence $n=4$, ;

\item {\em the general case}: one of the following holds

\begin{equation}
\begin{array}{l}
r=2 \mbox{ and } a_1\geq 5, \mbox{ or}\\
r=3 \mbox{ and } a_1\geq 3, \mbox{ or}\\
r=3, a_1=2 \mbox{ and } d\neq 4, \mbox{ or}\\
r>3 \mbox{ and } a_1\geq 2.
\end{array}
\label{generalcasecond}
\end{equation}

\end{itemize}

{\em The hyperplane case.} We need to study $CI(1,a_2,\ldots,
a_r)$ on the generic degree $d$ hypersurface of $\PP^{2r-2}$. As
one of the generators of the complete intersection is a
hyperplane, we can reduce to a smaller dimensional case. In
algebraic terms, for  a generic linear form $L$, we consider the
surjective quotient map
\[S\longrightarrow S/(L)\]
to get the following:

\begin{quote}{\it
the generic element of $S_d$ can be decomposed as a product of
forms of degrees $1,a_2,\ldots, a_r$, i.e. it has the form
$\sum_{i=1}^r F_iG_i$ with $\deg F_1=1$ and $\deg F_i=a_i$,
$i=2,\ldots,r$}
\end{quote}

if and only if

\begin{quote}{\it
the generic element of $(S/(L))_d\simeq(
\mathbb{C}[x_0,\ldots,x_{n-1}])_d$ can be decomposed as a product
of forms of degrees $a_2,\ldots, a_r$, i.e. it has the form
$\sum_{i=2}^r \bar{F_i}\bar{G_i}$ with $\deg \bar F_i=a_i$,
$i=2,\ldots,r$.}
\end{quote}

Hence, we have to study $CI(a_2,\ldots, a_r)$ on the generic
degree $d$ hypersurface of $\PP^{2r-3}$, i.e. codimension $r'=r-1$
complete intersections in $\PP^{n'}$, $n'=2r-3$. As $2r'=n'+1$
this situation was treated before and the only case where the
complete intersections exist is for $d=2$ and $a_i=1$ for all $i$.

{\em The plane case.} We have to study $CI(a_1,a_2)$ on the
generic degree $d$ curve of $\PP^2$ for $a_1=2,3,4$, any $a_2,d$
such that $a_1\leq a_2\leq d$. Now, $S=\CC[x_0,x_1,x_2]$ and we
consider forms $F_1,F_2,G_2$ and $G_1$ of degrees, respectively,
$a_1,a_2,d-a_2$ and $d-a_1$. We want to study the ring
\[A=S/(F_1,F_2,G_2)\]
and to compare $H(A,a_1)$ and $H(A,d)$ in order to apply
\eqref{six} of Lemma \ref{basiclemma} to show that
\[H(A/(\bar G_1),d)=0.\] Using Lemma \ref{basiclemma} \eqref{one}, we see that the last non-zero value of $H_A$
occurs for
\[ d+a_1-3.\]
In particular, for $a_1=2$, $H(A,d)=0$ and a $CI(2,a_2)$ exists
for any $a_2$ and $d$, $2\leq a_2\leq d$. If $a_1=3$, then
$H(A,d)=1$ and the same conclusion holds for $CI(3,a_2)$. Finally,
if $a_1=4$, then $H(A,d)=H(A,1)$ and it is easy to see that
$H(A,1)\leq H(A,4)$. Hence, for $a_1=2,3,4$ and any $a_2,d$ such
that $a_1\leq a_2\leq d$, the generic degree $d$ plane curve
contains a $CI(a_1,a_2)$.

{\em The four space case.} We address the case of $CI(2,2,2)$ on
the generic degree 4 threefold in $\PP^4$. Hence, we consider
$S=\CC[x_0,\ldots,x_4]$ and generic quadratic forms
$F_1,F_2,F_3,G_3,G_2$ and $G_1$. Let $A$ be the quotient ring
\[S/(F_1,F_2,F_3,G_3,G_2)\]
and notice that by the vanishing of the lefthand side of
\eqref{star} in the proof of Lemma \ref{basiclemma} we have that
$H(A,2)=H(A,3)$. Applying \eqref{two} of Lemma \ref{basiclemma}
yields $H(A,2)=H(A,4)$ and hence the required $CI$ exists.

{\em The general case.} Consider the ring
\[A=\CC[x_0,\ldots,x_{2r-2}]/(F_1,\ldots,F_r,G_r,\ldots,G_3,G_2)\]
and the multiplication map given by the degree $d-a_1$ form $\bar
G_1$
\[m:A_{a_1}\rightarrow A_d.\]
Clearly, with this notation, one has that the generic degree $d$
hypersurface contains a $CI(a_1,\ldots,a_r)$ if and only if
\[H(S/I,d)=H(A/(\overline{G_1}),d)=0\]
and this is equivalent to the surjectivity of $m$. We also recall
that by Lemma \ref{basiclemma} \eqref{six} $m$ has maximal rank.
Hence, to study the surjectivity we only have to compare
$H(A,a_1)=\dim A_{a_1}$ and $H(A,d)=\dim A_d$.

When $d=2$, all the degrees $a_i$ are equal to 1, and this was
treated in the hyperplane case.

Now we consider the $d>2$ case and we let $c={(r-1)d+a_1-2r+1\over
2}$. If $d\leq c$ then $\dim A_{a_1}<\dim A_{d}$ by Lemma
\ref{basiclemma}\eqref{four}, thus $m$ cannot be surjective.
Standard computations yield
\[d\leq c \Leftrightarrow 2d\leq (r-1)d+a_1-2r+1 \Leftrightarrow 2+{5-a_1\over r-3}\leq d \mbox{ if } r>3\]
while for $r=2$ the inequality $d\leq c$ never holds. Thus we get
that, when one of the conditions \eqref{generalcasecond} holds,
$m$ is not surjective if
\[r>3 \mbox{ and } d>7\]
or
\[r=3 \mbox{ and } a_1>5. \]

In the case $c<d$ we have to be more careful and the distances
$\alpha =c-a_1$ and $\beta =d-c$ have to be considered. When one
of the conditions \eqref{generalcasecond} holds, by
\ref{basiclemma}\eqref{five}, $m$ is surjective if and only if
$\alpha\leq \beta$. Thus we solve the inequality $\beta-\alpha\geq
0$. This is equivalent to
\[d\leq 2+{3\over r-2}\mbox{ if }r>2\]
while for $r=2$ we always have $\beta -\alpha\geq 0$.

Summing up all these facts one gets that when one of the
conditions \eqref{generalcasecond} holds:
\[r=2: m \mbox{ is surjective},\]
\[r=3: d>5, m \mbox{ is not surjective};d\leq 5, m \mbox{ is surjective};\]
\[r=4,5: d>3, m \mbox{ is not surjective};d\leq 3, m \mbox{ is surjective};\]
\[r\geq 6: d>2, m \mbox{ is not surjective};d\leq 2, m \mbox{ is surjective}\]
and using the treatment of the hyperplane, the plane and the four
space cases we obtain the final result.

\end{proof}

\begin{rem} The fact that a general hypersurface of degree $d\geq 6$ in $\PP^4$ cannot
contain a complete intersection of any type with $a_1,a_2,a_3<d$
is also a consequence of a result about vector bundles proved by
Mohan Kumar, Rao and Ravindra (see \cite{MKRR}).

In $\PP^n$ the existence statement for $d =  2$ is classical. The
$d=3$ cases in $\PP^6$ and $\PP^8$ can be obtained using Theorem
12.8 in \cite{Harris} (see also Proposition \ref{harrisPROP} in
this paper). In $\PP^4$, for $d=3$ and $d=4$ the existence also
follows from the analysis of arithmetically Cohen-Macaulay rank
two bundles on hypersurfaces, contained in \cite{ArrCos} and
\cite{Mad}.

In $\PP^4$, for the case $d=5$, when $\min\{a_i\}=2$, the result also follows from the existence of a canonical curve on the generic quintic threefold of $\PP^4$, as was essentially proved in \cite{Kl}.
\end{rem}

\begin{rem}
For $n=2$ the theorem above states that the generic degree $d$
plane curve contains a $CI(a,b)$ for any $a,b<d$, but it does not
say that this is a set of $ab$ points. The complete intersection
scheme could very well not be reduced. Actually, we can show
reducedness and hence the following holds
\begin{quote}{\em
the generic degree $d$ plane curve contains $ab$ complete
intersection points for any $a,b<d$. }\end{quote}
\end{rem}

\begin{rem}
Again a remark in the case $n=2$. Taking $a_1=a_2=a$ the theorem
above states that $\mbox{Sec}_1(\mathbb{X}_{(a,d-a)})$ is the
whole space. Now, quite generally, the points of the variety of
secant lines either lie on a true secant line or on a tangent line
to $\X_{(a,d-a)}$.  We claim that Proposition \ref{tangentspace}
allows us to conclude that the points of the tangent lines are in
fact already on the true secant lines.  In fact, if
$p=[FG]\in\mathbb{X}_{(a,d-a)}$, then any point $q$ of a tangent
line to $p$ can be written as $[\alpha FG'+\beta F'G]$ for forms
$G',F'$ of degrees $d-a,a$ respectively, and scalars $\alpha$ and
$\beta$. Thus, $q$ lies on the secant line to
$\mathbb{X}_{(a,d-a)}$ joining $[FG']$ and $[GF']$. In conclusion,
we can rephrase the equality
$\mbox{Sec}_1(\mathbb{X}_{(a,d-a)})=\PP(S_d)$ in terms of
polynomial decompositions
\begin{quote}{\em
let $a<d$, \underline{any} degree $d$ form in three variables $F$
can be written as $F=F_1G_1+F_2G_2$ for suitable forms $F_i$ of
degree $a$ and $G_i$ of degree $d-a$.}\end{quote}

This answers a question raised during a correspondence between
Zinovy Reichstein and the first author.
\end{rem}

\begin{rem}\label{hochesterRem} The restriction $2r\leq n+2$ in Theorem
\ref{maintheorem} is related to the fact that Fr\"{o}berg's
conjecture is only known to be true, in general, when the number
of forms does not exceed one more than the number of variables.
However, there are other partial results on this conjecture that
we can use to extend our theorem. E.g. in \cite{HoLak} Hochester
and Laksov showed that a piece of Fr\"{o}berg's conjecture holds.
More precisely they showed that if an ideal is generated by
generic forms of the same degree $d$ then the size of that ideal
in degree $d+1$ is exactly what is predicted by Fr\"{o}berg's
conjecture. Using this we can prove the following:
\begin{prop}\label{harrisPROP}
The generic degree $d>2$ hypersurface in $\PP^n$ contains a
complete intersection of type $(a_1,\ldots,a_r)$ where $a_i=1$ or
$a_i=d-1$ for all i, if and only if
\[{n-r+d\choose d}\leq (n-r+1)r.\]
When $a_1=\ldots=a_r=1$ this is the well known result on the
non-emptyness of the Fano variety of $(n-r)$-planes on the generic
degree $d$ hypersurface of $\PP^n$ (e.g. see \cite{Harris} Theorem
12.8).
\end{prop}
\begin{proof}
Using Lemma \ref{algeqlem} we have to show the vanishing, in
degree $d$, of the Hilbert function  of
\[A=\CC[x_0,\ldots,x_n]/(L_1,\ldots,L_r,F_1,\ldots,F_r)\]
for generic linear forms $L_i$ and generic degree $d-1$ forms
$F_i$. Clearly, as the linear forms are generic, we have
\[A\simeq\CC[x_0,\ldots,x_{n-r}]/(\overline{F_1},\ldots,\overline{F_r}).\]
Hence $A_d=0$ if and only if
$(\overline{F_1},\ldots,\overline{F_r})$ contains all the degree
$d$ forms. Using the result by Hochester and Laksov this is
equivalent to
\[{n-r+d\choose d}\leq (n-r+1)r.\]
\end{proof}

\end{rem}

\begin{ex} The variety $\X_{(1,3)}$ of reducible quartic hypersurfaces of $\PP^3$ and its secant line variety provide interesting examples for several reasons.

First note that $\X_{(1,3)} \subset \PP^{34}$ is a variety of
dimension $3 + 19 = 22$.  From Corollary \ref{tangentjoin} it is
easy to deduce that $\dim Sec_1(\X_{(1,3)}) = 33$.  Thus,
$\X_{(1,3)}$ is a defective variety whose virtual defect $e$ is,
$$e = 2\dim \X_{(1,3)} + 1 - \dim Sec_1(\X_{(1,3)}) = 12.$$

\medskip\noindent $1)$ \ Consider the Noether-Lefschetz locus of quartic hypersurfaces in $\PP^3$ with Picard group $\neq \Z$.  The quartic hypersurfaces which contain a line are clearly in the Noether-Lefschetz locus.  If $\ell$ is a line defined by the linear forms $L_1, L_2$ then the form $F$, of degree 4, defines a hypersurface containing $\ell$ if and only if
$$
F = L_1G_1 + L_2G_2 \hbox{ where } \deg G_i = 3 .
$$

I.e. if and only if $[F] \in Sec_1(\X_{(1,3)})$.  Since, as we observed, $$\dim Sec_1(\X_{(1,3)}) = 33$$ this forces the secant variety to be a component of the Noether-Lefschetz locus.

We wonder how often joins of other varieties of reducible forms give components of the appropriate Noether-Lefschetz locus.

\medskip\noindent $2)$ \ Since $\X_{(1,3)}$ is defective for secant lines we have, by a theorem of \cite{ChCi06}, that for every two points on $\X_{(1,3)}$ there is a subvariety $\Sigma$,  containing those two points, whose linear span has dimension $ \leq 2 \dim \Sigma + 1 - e$, where $e$ is the defect of $\X_{(1,3)}$.

We now give a description of such $\Sigma$'s for the variety $\X_{(1,3)}$.

Let $[H_1F_1], [H_2F_2]$ be two points of $\X_{(1,3)}$ and let
$\ell$ be the line in $\PP^3$ defined by $H_1 = 0 = H_2$. Consider
$\Sigma \subset \X_{(1,3)}$, the subvariety of reducible quartics
whose linear components contain $\ell$.  Clearly $\dim \Sigma = 1
+ 19 = 20$.  Notice that the linear span of $\Sigma$, $<\Sigma>$,
is contained in the subvariety of all quartics containing $\ell$
and that variety has dimension $34 - 5 = 29$.

Thus, $$\dim <\Sigma > \leq 29 = 2(20) + 1 - 12 = 2 \dim \Sigma + 1 - e,$$ as we wanted to show.

Notice that the existence of $\Sigma$, as above, gives another proof of the defectivity of $\X_{(1,3)}$.

\end{ex}

\bibliographystyle{alpha}
\bibliography{carlini}

\def\cprime{$'$}
\begin{thebibliography}{MKRR06}

\bibitem[AC00]{ArrCos}
Enrique Arrondo and Laura Costa.
\newblock Vector bundles on {F}ano 3-folds without intermediate cohomology.
\newblock {\em Comm. Algebra}, 28(8):3899--3911, 2000.

\bibitem[{\AA}dl88]{AAd}
B.~{\AA}dlandsvik.
\newblock Varieties with an extremal number of degenerate higher secant
  varieties.
\newblock {\em J. Reine Angew. Math.}, 392:16--26, 1988.

\bibitem[Ani86]{Anick}
David~J. Anick.
\newblock Thin algebras of embedding dimension three.
\newblock {\em J. Algebra}, 100(1):235--259, 1986.

\bibitem[Car05]{Ca05Siena}
Enrico Carlini.
\newblock Codimension one decompositions and {C}how varieties.
\newblock In {\em Projective varieties with unexpected properties}, pages
  67--79. Walter de Gruyter GmbH \& Co. KG, Berlin, 2005.

\bibitem[Car06]{Ca04JA}
Enrico Carlini.
\newblock Binary decompositions and varieties of sums of binaries.
\newblock {\em J. Pure Appl. Algebra}, 204(2):380--388, 2006.

\bibitem[CC06]{ChCi06}
Luca Chiantini and Ciro Ciliberto.
\newblock On the concept of {$k$}-secant order of a variety.
\newblock {\em J. London Math. Soc. (2)}, 73(2):436--454, 2006.

\bibitem[CGG02]{CGG1}
M.~V. Catalisano, A.~V. Geramita, and A.~Gimigliano.
\newblock Ranks of tensors, secant varieties of {S}egre varieties and fat
  points.
\newblock {\em Linear Algebra Appl.}, 355:263--285, 2002.

\bibitem[CGG03]{CGG2}
M.~V. Catalisano, A.~V. Geramita, and A.~Gimigliano.
\newblock Erratum of the publisher to: ``{R}anks of tensors, secant varieties
  of {S}egre varieties and fat points'' [{L}inear {A}lgebra {A}ppl.\ {\bf 355}
  (2002), 263--285; {MR}1930149 (2003g:14070)].
\newblock {\em Linear Algebra Appl.}, 367:347--348, 2003.

\bibitem[Chi02]{Chi02}
Jaydeep~V. Chipalkatti.
\newblock Decomposable ternary cubics.
\newblock {\em Experiment. Math.}, 11(1):69--80, 2002.

\bibitem[Cil01]{Ci01}
C.~Ciliberto.
\newblock Geometric aspects of polynomial interpolation in more variables and
  of {W}aring's problem.
\newblock In {\em European Congress of Mathematics, Vol. I (Barcelona, 2000)},
  volume 201 of {\em Progr. Math.}, pages 289--316. Birkh\"auser, Basel, 2001.

\bibitem[CJ96]{CatJ}
Michael~L. Catalano-Johnson.
\newblock The possible dimensions of the higher secant varieties.
\newblock {\em Amer. J. Math.}, 118(2):355--361, 1996.

\bibitem[Fr{\"o}85]{Fro}
Ralf Fr{\"o}berg.
\newblock An inequality for {H}ilbert series of graded algebras.
\newblock {\em Math. Scand.}, 56(2):117--144, 1985.

\bibitem[Ger96]{Ge}
A.~V. Geramita.
\newblock Inverse systems of fat points: {W}aring's problem, secant varieties
  of {V}eronese varieties and parameter spaces for {G}orenstein ideals.
\newblock In {\em The Curves Seminar at Queen's, Vol.\ X (Kingston, ON, 1995)},
  volume 102 of {\em Queen's Papers in Pure and Appl. Math.}, pages 2--114.
  Queen's Univ., Kingston, ON, 1996.

\bibitem[GH85]{GH85}
Phillip Griffiths and Joe Harris.
\newblock On the {N}oether-{L}efschetz theorem and some remarks on
  codimension-two cycles.
\newblock {\em Math. Ann.}, 271(1):31--51, 1985.

\bibitem[Gro05]{SGA2}
Alexander Grothendieck.
\newblock {\em Cohomologie locale des faisceaux coh\'erents et th\'eor\`emes de
  {L}efschetz locaux et globaux ({SGA} 2)}.
\newblock Documents Math\'ematiques (Paris) [Mathematical Documents (Paris)],
  4. Soci\'et\'e Math\'ematique de France, Paris, 2005.
\newblock S\'eminaire de G\'eom\'etrie Alg\'ebrique du Bois Marie, 1962,
  Augment\'e d'un expos\'e de Mich\`ele Raynaud. [With an expos\'e by Mich\`ele
  Raynaud], With a preface and edited by Yves Laszlo, Revised reprint of the
  1968 French original.

\bibitem[GS98]{GeSc}
Anthony~V. Geramita and Henry~K. Schenck.
\newblock Fat points, inverse systems, and piecewise polynomial functions.
\newblock {\em J. Algebra}, 204(1):116--128, 1998.

\bibitem[Har92]{Harris}
J~Harris.
\newblock {\em Algebraic geometry, A first course}.
\newblock Graduate Texts in Math. Springer-Verlag, New York, 1992.

\bibitem[HL87]{HoLak}
Melvin Hochster and Dan Laksov.
\newblock The linear syzygies of generic forms.
\newblock {\em Comm. Algebra}, 15(1-2):227--239, 1987.

\bibitem[IK99]{IaKa}
A.~Iarrobino and V.~Kanev.
\newblock {\em Power sums, {G}orenstein algebras, and determinantal loci},
  volume 1721 of {\em Lecture Notes in Mathematics}.
\newblock Springer-Verlag, Berlin, 1999.

\bibitem[Kle00]{Kl}
Holger~P. Kley.
\newblock Rigid curves in complete intersection {C}alabi-{Y}au threefolds.
\newblock {\em Compositio Math.}, 123(2):185--208, 2000.

\bibitem[Lef21]{Lefschetz}
Solomon Lefschetz.
\newblock On certain numerical invariants of algebraic varieties with
  application to abelian varieties.
\newblock {\em Trans. Amer. Math. Soc.}, 22(3):327--406, 1921.

\bibitem[LM04]{LM}
J.~M. Landsberg and L.~Manivel.
\newblock On the ideals of secant varieties of {S}egre varieties.
\newblock {\em Found. Comput. Math.}, 4(4):397--422, 2004.

\bibitem[Mad00]{Mad}
Carlo Madonna.
\newblock Rank-two vector bundles on general quartic hypersurfaces in {${\Bbb
  P}\sp 4$}.
\newblock {\em Rev. Mat. Complut.}, 13(2):287--301, 2000.

\bibitem[Mam54]{Mamma}
Carmelo Mammana.
\newblock Sulla variet\`a delle curve algebriche piane spezzate in un dato
  modo.
\newblock {\em Ann. Scuola Norm. Super. Pisa (3)}, 8:53--75, 1954.

\bibitem[MKRR06]{MKRR}
N.~Mohan~Kumar, A.~P. Rao, and G.~V. Ravindra.
\newblock Four-by-four {P}faffians.
\newblock {\em Rend. Semin. Mat. Univ. Politec. Torino}, 64(4):471--477, 2006.

\bibitem[MMR03]{MigMR}
J.~Migliore and R.~M. Mir{\'o}-Roig.
\newblock Ideals of general forms and the ubiquity of the weak {L}efschetz
  property.
\newblock {\em J. Pure Appl. Algebra}, 182(1):79--107, 2003.

\bibitem[Sev06]{Severi}
F.~Severi.
\newblock Una proprieta' delle forme algebriche prive di punti muiltipli.
\newblock {\em Rend. Accad. Lincei, II}, 15:691--696, 1906.

\bibitem[Sta80]{Stanley}
Richard~P. Stanley.
\newblock Weyl groups, the hard {L}efschetz theorem, and the {S}perner
  property.
\newblock {\em SIAM J. Algebraic Discrete Methods}, 1(2):168--184, 1980.

\bibitem[Wat87]{Wat}
Junzo Watanabe.
\newblock The {D}ilworth number of {A}rtinian rings and finite posets with rank
  function.
\newblock In {\em Commutative algebra and combinatorics (Kyoto, 1985)},
  volume~11 of {\em Adv. Stud. Pure Math.}, pages 303--312. North-Holland,
  Amsterdam, 1987.

\end{thebibliography}

\end{document}